\documentclass[10pt,reqno]{amsart}

\usepackage{amsmath, amsfonts, amssymb, latexsym, amsthm}
\usepackage[pagewise]{lineno}
\usepackage{amsmath, amsfonts, amssymb, latexsym}
\usepackage[numbers,sort&compress]{natbib}
\usepackage{mathrsfs}
\usepackage{pdfcomment}
\hypersetup{hidelinks,
	colorlinks=true,
	allcolors=black,
	pdfstartview=Fit,
	breaklinks=true
}

\everymath{\displaystyle}

\newtheorem{theorem}{Theorem}
\theoremstyle{plain}

\newtheorem{lemma}[theorem]{Lemma}

\newtheorem{proposition}[theorem]{Proposition}

\newtheorem{remark}[theorem]{Remark}

\numberwithin{equation}{section}
\numberwithin{theorem}{section}

\newcommand{\cU}{\mathcal{U}}

\newcommand{\cE}{\mathcal{E}}

\newcommand{\cL}{\mathcal{L}}
\newcommand{\cJ}{\mathcal{J}}

\newcommand{\cF}{\mathcal{F}}

\newcommand{\E}{\mathbb{E}}
\newcommand{\R}{\mathbb{R}}
\newcommand{\F}{\mathbb{F}}

\newcommand{\bP}{\mathbb{P}} 

\def\wt{\widetilde}

\def\Ito{It\^{o} }

\begin{document}
\title
{Impulse approximate controllability for stochastic evolution equations and its applications}

\author{\sffamily Yuanhang Liu$^{1,^*}$   \\
	{\sffamily\small $^1$ School of Mathematics and Statistics, Central South University, Changsha 410083, China. }
}
	\footnotetext[1]{Corresponding author: liuyuanhang97@163.com}



\maketitle

\begin{abstract}
This paper is concerned with impulse approximate controllability for stochastic evolution equations with impulse controls. As direct applications, we formulate captivating minimal norm and time optimal control problems; The minimal norm problem seeks to identify an optimal impulse control characterized by the minimum norm among all feasible controls, guiding the system's solutions from an initial state within a fixed time interval toward a predetermined target while the minimal time problem is to find an optimal impulse control (among certain control constraint set), which steers the solution of the stochastic equation from a given initial state to a given target set as soon as possible.  These problems, to the best of our knowledge, are among the first to discuss in the stochastic case.
\end{abstract}

\pagestyle{myheadings}
\thispagestyle{plain}
\markboth{Impulse approximate controllability and its applications}{YUANHANG LIU}

\section{Introduction}
Impulse control, situated within a significant realm of control theory, boasts extensive applications (refer, for instance, to \cite{bellman1971topics,bensoussan1975controle,tsien1954engineering,yang2001impulsive}). In numerous scenarios, impulse control emerges as a compelling alternative for addressing systems resistant to continuous control inputs. Consider, for example, the control of a bacterial population; employing impulse control allows for instantaneous changes in bactericide density. In contrast, continuous control mechanisms risk bolstering bacterial drug resistance (cf. \cite{trelat2016impulse,yang2001impulsive}). In \cite{yang2001impulsive}, the author systematically elucidates the theory and applications of impulse controlled ordinary differential equations. The exposition not only delves into the foundational principles but also provides insights into the practical implementation of impulse controls. The realm of optimal control and controllability for impulse controlled equations has garnered considerable attention in the literature. In \cite{duan2020minimal}, the authors delved into a specific minimal norm optimal control problem concerning a semilinear heat equation integrated with impulse controls. This problem involves identifying an impulse control solution characterized by the minimal norm, directing the controlled system from an initial state to a predefined target within a fixed time interval. In \cite{duan2018minimal}, a different perspective emerged, focusing on a minimal time control problem applied to a linear evolution equation incorporating impulse controls. This problem revolves around two crucial parameters: the upper limit of control constraints and the timing of impulse occurrences. The objective here is to identify an optimal impulse control, within a specified control constraint set, that swiftly guides the evolution equation's solution from an initial state towards a predefined target set. Additionally, in \cite{yong1992necessary}, the authors investigated optimal control scenarios for semilinear evolutionary distributed parameter systems featuring impulse controls. Their study derived necessary optimal control conditions, extending the conventional Pontryagin's maximum principle. For further exploration in the domain of impulse control problems, see \cite{qishu2016periodic,phung2017impulse,buffe2018spectral} and references therein.

Note that the previous research mentioned focused on deterministic equations. However, in practice, the consideration of stochastic effects requires replacing deterministic functions with stochastic processes as mathematical descriptions, leading to the formation of stochastic equations. In \cite{belak2017general}, the authors  established the existence of optimal controls within the framework of a comprehensive stochastic impulse control problem. The key insight lies in characterizing the value function as the pointwise minimum among a set of superharmonic functions. Furthermore, it emerges as the unique continuous viscosity solution to quasi-variational inequalities and as the limit of a sequence of iterated optimal stopping problems. Employing a synthesis of these characterizations, they construct optimal controls, avoiding reliance on any regularity of the value function beyond its continuity. In \cite{basei2022nonzero}, a broader examination unfolds, encompassing a general class of nonzero-sum N-player stochastic games featuring impulse controls, where players exert discrete interventions to manipulate the underlying dynamics. Employing a verification approach, the authors present sufficient conditions governing the Nash equilibria of the game. Notably, their study includes a detailed numerical analysis spanning scenarios such as the single-player case, the two-player game, and the mean-field game with impulse controls. This comprehensive exploration unveils the profound impact of competition on the optimal strategy of players, accompanied by a sensitivity analysis of the model parameters. In \cite{korn2017stochastic}, the authors delved into the intricacies of optimal product management, particularly in the context of multiple product generations evolving continuously over time. This investigation naturally gave rise to the exploration of dynamic optimal control problems, characterized by the inclusion of intervention costs and partially controlled regime shifts. Consequently, the study focuses on the analysis and resolution of stochastic impulse control problems featuring regime-switching within a broad and general framework. For further exploration in the domain of impulse control problems for stochastic equations, see \cite{aid2020nonzero,el2020zero,el2020stochastic} and references therein.

Thus far, there has been no research on the impulse approximate controllability for controlled stochastic systems. This paper presents the first attempt to address this problem for stochastic systems. As direct applications, we formulate captivating minimal norm and time optimal control problems; The minimal norm problem seeks to identify an optimal impulse control characterized by the minimum norm among all feasible controls, guiding the system's solutions from an initial state within a fixed time interval toward a predetermined target while the minimal time problem is to find an optimal impulse control (among certain control constraint
set), which steers the solution of the stochastic equation from a given initial state to a
given target set as soon as possible. Compared to the deterministic case, stochastic terms arise when studying the impulse approximate controllability, which complicates further study of the norm optimal  control problem. To overcome this difficulty, we adopt relevant techniques from \cite{lv1,liu2023observability} to establish the the impulse approximate controllability and to study the minimal norm control problem. On the other hand, it is important to recognize that we cannot employ the time change technique and treat the backward and forward equations in the same manner as in the deterministic case. This is because the stochastic system requires adaptedness, which cannot be disregarded in calculations. The adaptedness of stochastic processes has emerged as a crucial hindrance in exploring the impulse approximate controllability. Notably, the solutions of the forward and backward deterministic parabolic equations are
equivalent to show the impulse approximate controllability. However, for stochastic equations, the solutions of these equations are not
equivalent. This is the main difficulty that we shall overcome in this paper.

The rest of the paper is organized as follows. In Section 2, the main problem is formulated and the main results Theorem \ref{thm:main}, Theorem \ref{thm:main2} and Theorem \ref{impulse-4} are stated. In Section 3, some auxiliary results to be used later are presented. The proof of Theorem \ref{thm:main} is provided in Section 4. In Section 5 and 6, we delve into the discussions of forthcoming norm and time optimal control problems
(to be formulated later) and provide the proofs for Theorem \ref{thm:main2} and Theorem \ref{impulse-4}, respectively. Finally, in Section 7, we show some specific evolution equations, which are covered by the abstract framework of this paper.
\section{Problem formulation and main results}
At first, let us introduce necessary notations.

Let $(\Omega,\mathcal{F},\left\{\mathcal{F}_t\right\}_{t\geq 0}, \bP)$ be a fixed complete filtered probability space, on which a one dimensional standard Brownian motion $\{W(t)\}_{t\geq0}$ is defined, and $\left\{\mathcal{F}_t\right\}_{t\geq 0}$ is the corresponding natural filtration, augmented by all the $\bP$-null sets in $\mathcal{F}$. We denote by $\mathbb{F}$ the progressive $\sigma$-field w.r.t. $\left\{\mathcal{F}_t\right\}_{t\geq 0}$.

Let $H$ and $U$ be  two separable  Hilbert spaces with inner products $\langle\cdot,\cdot\rangle_H$ and $\langle\cdot,\cdot\rangle_U$;
and norms $\|\cdot\|_H$ and $\|\cdot\|_U$, respectively. Fix $t \ge 0, p \in [1,\infty)$, we denote by $L^p_{\mathcal{F}_t}(\Omega;H)$  the Banach space consisting of all $H$-valued,  $\mathcal{F}_t$ measurable random
variables $X(t)$ endowed with the norm
$$
\|X(t)\|_{L^p_{\mathcal{F}_t}(\Omega;H)}=\bigg(\mathbb{E}\|X(t)\|^p_{H}\bigg)^{\frac{1}{p}}.
$$
Denote by $L^p_{\mathbb{F}}(0,T;L^q(\Omega;H))$, $p,q\in[1,\infty)$, the Banach space consisting of all $H$-valued, $\left\{\mathcal{F}_t\right\}_{t\geq 0}$-adapted processes $X$ endowed with the norm
$$
\|X(\cdot)\|_{L^p_{\mathbb{F}}(0,T;L^q(\Omega;H))}=\bigg(\int_0^T(\mathbb{E}\|X(t)\|^q_H)^{\frac{p}{q}}dt\bigg)^{\frac{1}{p}}.
$$
Denote by $L^p_{\mathbb{F}}(\Omega; L^q(0,T;H))$, $p,q\in[1,\infty)$, the Banach space consisting of all $H$-valued, $\left\{\mathcal{F}_t\right\}_{t\geq 0}$-adapted processes $X$ endowed with the norm
$$
\|X(\cdot)\|_{L^p_{\mathbb{F}}(\Omega; L^q(0,T;H))}=\bigg[\mathbb{E}\bigg(\int_0^T\|X(t)\|^q_H dt\bigg)^{\frac{p}{q}}\bigg]^{\frac{1}{p}}.
$$
Denote by $L^\infty_{\mathbb{F}}(0,T;\R)$, the Banach space consisting of all $\R$-valued, $\left\{\mathcal{F}_t\right\}_{t\geq 0}$-adapted bounded processes, with the essential supremum norm.\\
Denote by $L^q_{\mathbb{F}}(\Omega;C([0,T];H))$, $q\in[1,\infty)$, the Banach space consisting of all $H$-valued, $\left\{\mathcal{F}_t\right\}_{t\geq 0}$-adapted continuous processes $X$ endowed with the norm
$$
\|X(\cdot)\|_{L^q_{\mathbb{F}}(\Omega;C([0,T];H))}=\bigg(\mathbb{E}\|X(\cdot)\|^q_{C([0;T];H)}\bigg)^{\frac{1}{q}}.
$$

In the sequel, we shall simply denote $$L^p_{\mathbb{F}}(0,T;H):= L^p_{\mathbb{F}}(\Omega; L^p(0,T;H))~\text{with}~p\in[1,\infty),$$
and denote by $|\cdot|$ the Lebesgue measure on $\R^n,n\geq1$.

Let $0<\tilde{T}<T\leq2\tilde{T}$. Consider the following forward controlled stochastic evolution equation:
\begin{equation}
\label{eq:main-for}
\left\{
\begin{array}{ll}
dy = A ydt + F(t)y dW(t), & t\in(0,T)\setminus\{\tilde{T}\},   \\[2mm]
y(0) =y_0,\,y(\tilde{T})=y(\tilde{T}^-)+Bu,
\end{array}
\right.
\end{equation}
where $y_0\in L^2_{\cF_0}(\Omega;H)$, the control variable $u\in L^2_{\cF_{\tilde{T}}}(\Omega;U)$,
$$y(\tilde{T}^-):=\lim_{t\uparrow\tilde{T}}y(t)~\text{in}~H,\,\bP-a.s.,$$
the operator $A$ generating a $C_0$-semigroup $\{S(t)\}_{t\geq0}$ on $H$ is a linear self-adjoint operator from $D(A)$ into $H$, where $D(A)$, being the domain of $A$, is a subspace of $H$, $B\in\cL(U,H)$ (the space of linear continuous operators from $U$ to $H$) and denote by $B^*\in \cL(H,U)$ the adjoint operator of $B$ and $F\in L_{\F}^\infty(0,T;\R)$ is a given function. By the classical well-posedness result for forward stochastic evolution equations, see e.g., Section 3 in \cite{lv1},  we know that equation (\ref{eq:main-for}) admits a unique solution $y\in L^2_{\mathbb{F}}(\Omega;C([0,T];H))$.

Now, we write
\[
0<\lambda_1\le \lambda_2\le \cdots
\]
for the eigenvalues of $-A$, and $\{e_j\}_{j\ge 1}$ for corresponding eigenfunctions which form an orthonormal basis for $H$. For each $\lambda>0$, we define
\[
\cE_\lambda f = \sum_{\lambda_j\le\lambda} \langle f,e_j\rangle_H e_j,~\text{and}~
\cE^\perp_\lambda f = \sum_{\lambda_j>\lambda} \langle f,e_j\rangle_H e_j, ~\text{for each}~f \in H.
\]
Consequently, we have
\[
  f = \cE_\lambda f + \cE_\lambda^\perp f.
\]

Next, we introduce the following assumptions:

${\bf(H)}$: There are constants $\gamma\in(0,1)$ and $N>0$ such that for any $\lambda>0$,
$$
\|\cE_\lambda f\|_H\leq e^{N(1+\lambda^{\gamma})}\|B^*\cE_\lambda f\|_U,~\text{for all}~f\in H.
$$

${\bf(B)}$:
$$
\|B^*f\|_U\leq \|f\|_H,~\text{for all}~f\in H.
$$

${\bf(S)}$: Let a control sequence $\{u_n\}_{n\geq1}\subset L^2_{\cF_{T}}(\Omega;U)$ and solution $y_n(\cdot;y_0,u_n)$ be satisfied the following system 
$$
\left\{
\begin{array}{ll}
dy_n = A y_ndt + F(t)y_n dW(t), & D\times(0,T)\setminus\{\tilde{T}\},   \\[2mm]
y_n(0) =y_0,\,y_n(\tilde{T})=y_n(\tilde{T}^-)+Bu_n.
\end{array}
\right.
$$
If there is a $u\in L^2_{\cF_{T}}(\Omega;U)$ such that $$u_n\rightarrow u~\text{weakly in}~\,L^2_{\cF_{T}}(\Omega;U)~ \text{as}~\,n\rightarrow\infty,$$ then there is a $y\in L^2_{\mathbb{F}}(\Omega;C([0,T];H))$ solving \eqref{eq:main-for} such that $$y_n\rightarrow y~\text{strongly in}~C([0,T];H), \,\bP-a.s.,~ \text{as}~\,n\rightarrow\infty.$$

The following is our main result.
\begin{theorem}\label{thm:main}
Let $0<\tilde{T}<T\leq2\tilde{T}$. Suppose the assumptions ${\bf(H)}$ and ${\bf(B)}$ hold.  Then system \eqref{eq:main-for} is impulse approximate controllability. That is to say, for any $y_0 \in L^2_{\cF_0}(\Omega;H)$, $\epsilon>0$, the solution $y$ of \eqref{eq:main-for} satisfies the following inequality:
\begin{equation}\label{thm:main-ob}
\E\|y(T)\|_H^2 \leq \epsilon \E\|y_0\|_H^2.
\end{equation}
Moreover, there exist two constants $C_3>0,C_4>0$ such that the control $u\in L^2_{\cF_{\tilde{T}}}(\Omega;U)$ satisfies
\begin{equation}\label{thm:main-u}
\E\|u\|_U^2 \leq e^{C_4\left[1+\left(\frac{1}{T-\tilde{T}}\right)^\frac{\gamma}{1-\gamma}\right]}e^{\left[\frac{C_3}{T-\tilde{T}}\ln\left(e+\frac{1}{\epsilon}\right)\right]^\gamma} \E\|y_0\|_H^2.
\end{equation}
\end{theorem}

In order to prove Theorem \ref{thm:main}, let us introduce the following backward stochastic evolution equation:
\begin{equation}
\label{eq:main}
\left\{
\begin{array}{ll}
d z = -A zdt - F(t)Z dt + Z dW(t), & t\in(0,T),   \\[2mm]
z(T) =\eta,
\end{array}
\right.
\end{equation}
where $\eta\in L^2_{\mathcal{F}_T}(\Omega;H)$.
We denote by $(z(\cdot;\eta), Z(\cdot;\eta))$ the solution of equation
\eqref{eq:main} given the terminal condition $\eta=z(T)$.

Based on Theorem \ref{thm:main}, we consider the following minimal norm control problem ${\bf (NP)^T_{y_0}}$: Given $\epsilon>0$,
\begin{equation}\label{N}
N(T,y_0):=\{\E\|u\|_U^2: \E\|y(T)\|_H^2 \leq \epsilon \E\|y_0\|_H^2\},
\end{equation}
where $y$ is the solution of \eqref{eq:main-for}.

In problem ${\bf(NP)}_{y_0}^T$, a control $u \in L^2_{\cF_{T}}(\Omega;U)$ is called admissible, if $\E\|y(T)\|_H^2 \leq \epsilon \E\|y_0\|_H^2$ and a control $u^* \in L^2_{\cF_{T}}(\Omega;U)$ is called optimal, if it is admissible and $\E\|u^*\|_U^2= N(T, y_0)$, which is the optimal norm.

The second main result is to address the norm control problem ${\bf (NP)^T_{y_0}}$.
\begin{theorem}\label{thm:main2}
Suppose the assumptions ${\bf(H)}$, ${\bf(B)}$ and ${\bf(S)}$ hold. For any $y_0 \in L^2_{\cF_0}(\Omega;H)$ and given $\epsilon>0$, problem ${\bf(NP)}_{y_0}^T$ has a unique optimal control.
\end{theorem}
Next, consider the following forward controlled stochastic evolution equation:
\begin{equation}
\label{eq:main-for2}
\left\{
\begin{array}{ll}
dy = A ydt + F(t)y dW(t), & D\times(0,+\infty)\setminus\{\tilde{T}\},   \\[2mm]
y(0) =y_0,\,y(\tilde{T})=y(\tilde{T}^-)+Bu.
\end{array}
\right.
\end{equation}

For each $M>0$, we define a control constraint set $\mathcal{U}_M$ as follows:
\begin{equation*}
\mathcal{U}_M\triangleq \{u\in L^2_{\cF_{T}}(\Omega;U): \|u\|_{L^2_{\cF_{T}}(\Omega;U)}\leq M\}.
\end{equation*}
Given $\tilde{T}>0$, $\epsilon>0$ and $M>0$, we now consider the minimal time control problem
\begin{equation}\label{impulse-2(1)}
{\bf (TP)}:\;\;\;\;T^*\triangleq \inf_{u\in \mathcal{U}_M} \{T>\tilde{T}: y(T; y_0, u)\in B_\epsilon(0)\},
\end{equation}
where
$$
B_\epsilon(0):=\{f\in L^2_{\cF_{T}}(\Omega;H):\E\|f\|_H^2 \leq \epsilon \E\|y_0\|_H^2\},
$$
and $y(\cdot;y_0,u)$ is the solution of \eqref{eq:main-for2}.

For this problem ${\bf (TP)}$, we say that $B_\epsilon(0)$ is the target set, and that $u\in \mathcal{U}_M$ is an admissible control if there exists a $T>\tilde{T}$ such that $y(T; y_0, u)\in B_\epsilon(0)$;  we denote by $T^*$ the optimal time if it exists, and by $u^*\in \mathcal{U}_M$ an optimal control if $T^*>\tilde{T}$ such that $y(T^*; y_0, u^*)\in B_\epsilon(0)$.

The third main result is to address the time optimal control problem ${\bf (TP)}$.
\begin{theorem}\label{impulse-4}
Suppose the assumptions ${\bf(H)}$, ${\bf(B)}$ and ${\bf(S)}$ hold. Let $\tilde{T}>0$, $M>0$, $\epsilon> 0$ and $\eta\in L^2_{\cF_T}(\Omega;H)$ with $\eta\neq0$. Then, the time optimal control problem ${\bf (TP)}$ admits a solution, i.e., there exists a unique optimal control $u^*\in \mathcal{U}_M$ such that when $T^*>\tilde{T}$, we have $y(T^*; y_0, u^*)\in B_\epsilon(0)$. Moreover, the optimal control is defined as
$$
u^*=M\frac{B^* z(\tilde{T})}{\left(\E\|B^* z(\tilde{T})\|_U^2\right)^{\frac{1}{2}}},
$$
where $z(\tilde{T};\eta)$ is the solution of \eqref{eq:main} at time $\tilde{T}$,
and it satisfies
\begin{equation}\label{tu}
\E\|u^*\|_U^2=M^2.
\end{equation}
\end{theorem}
Several Remarks are given in order.
\begin{remark}
The system \eqref{eq:main-for2} can be understood as follows: Over $(0,T)$, the system \eqref{eq:main-for2} with an impulse control $u \in L^2_{\cF_{\tilde{T}}}(\Omega;U)$ is controlled, which is the same as system \eqref{eq:main-for}. Then, over $(T,+\infty)$, we let the system \eqref{eq:main-for2} without control freely evolve.
\end{remark}
\begin{remark}
In the event that no admissible control exists for ${\bf (NP)^T_{y_0}}$ and ${\bf (TP)}$, we say that $\inf\varnothing=+\infty$. However, it is noteworthy to emphasize that the sets of admissible controls for ${\bf (NP)^T_{y_0}}$ and ${\bf (TP)}$ are, in fact, not empty based on Theorem \ref{thm:main}.
\end{remark}
\begin{remark}
In Theorem \ref{thm:main}, we make an assumption $0<\tilde{T}<T\leq2\tilde{T}$. This is because in the proof of the impulse approximate controllability of system \eqref{eq:main-for}, we construct a control $u\in L^2_{\cF_{T-\tilde{T}}}(\Omega;U)$, which along with $y(\tilde{T})=y(\tilde{T}^-)+Bu$ in system \eqref{eq:main-for} and the adaptedness of $y(\tilde{T})$, it implies that  we should assume $0<\tilde{T}<T\leq2\tilde{T}$. For a detailed exposition, please refer to Theorem \ref{th-1}.
\end{remark}
\begin{remark}
In Theorem \ref{impulse-4}, we suppose $\eta\neq0$. This is because by the claim \eqref{claim} below and the assumption ${\bf(B)}$, we see that $\E\|B^*z(t)\|_U^2\neq0$. Then the optimal control $u^*$ defined in Theorem \ref{impulse-4} is meaningful. However, in Theorem \ref{thm:main} and Theorem \ref{thm:main2}, we do not suppose $\eta\neq0$.
\end{remark}

\begin{remark}
We shall present some specific equations which are covered by \eqref{eq:main-for} and satisfy the assumption ${\bf (H)}$, ${\bf (B)}$ and ${\bf (S)}$, please see Section 7 for more details.
\end{remark}

\section{Auxiliary conclusions}
In this section, we give some auxiliary results that will be used later. At first, by linearity for \eqref{eq:main}, it is easy to check that
\begin{align}
  & z(t; \cE_\lambda \eta) = \sum_{\lambda_j\le \lambda}
  z_j(t;\eta_j)e_j = \cE_\lambda z(t;\eta);\nonumber\\
& z(t; \cE^\bot_\lambda \eta) = \sum_{\lambda_j> \lambda}
  z_j(t; \eta_j)e_j = \cE^\bot_\lambda z(t;\eta), \label{eq:2}
\end{align}
where $\eta_j=\langle\eta,e_j\rangle_H$ and $(z_j(\cdot;\eta_j), Z_j(\cdot;\eta_j))$ is the solution of the following backward
stochastic differential equation
\begin{equation}
\label{eq:main-1}
\left\{
\begin{array}{ll}
d z_j = \lambda_j z_jdt + F(t)Z_j(t) dt + Z_j dW(t), & t\in(0,T),   \\[2mm]
z_j(T) =\eta_j.
\end{array}
\right.
\end{equation}
Set
$$
\tau=\|F\|_{L_{\F}^\infty(0,T;\R)}^2.
$$

Now recall an important result used later in this paper:
\begin{lemma}\cite[Lemma 2.3]{liu2023observability}
\label{decay}
Given any $\eta$ in the space of  $L^2_{\mathcal{F}_T}(\Omega;H)$,
    we have for each $t\in [0,T]$, the solution of \eqref{eq:main} satisfies
  \begin{equation}
    \label{eq:decay}
    \E \|z(t ;\cE_\lambda^\bot \eta)\|_H^2 \leq e^{(-2\lambda+\tau)(T-t)} \E \|\eta\|_H^2.
  \end{equation}
\end{lemma}

\begin{lemma}\label{lemma:main}
Given $T>0$. Suppose $E\subset(0,T)$ is a measurable set with positive measure.  Then there exists a positive constant $C$ such that the solution $z$ to the equation \eqref{eq:main} satisfies the following observability inequality
\begin{equation}\label{lemma:main-ob}
\E\|z(0;\eta)\|_H^2 \leq C \|\chi_EB^*z(\cdot;\eta)\|^2_{L^2_\F(0,T;U)}.
\end{equation}
\end{lemma}
\begin{proof}
From Theorem 2.1 in \cite{liu2023observability}, we can obtain \eqref{lemma:main-ob}.
\end{proof}

Next, based on Lemma \ref{decay} and Lemma \ref{lemma:main}, we provide an interpolation inequality for system \eqref{eq:main}.
\begin{lemma}
\label{lemma-1}
Suppose the assumptions ${\bf(H)}$ and ${\bf(B)}$ hold. Given $\eta\in L^2_{\mathcal{F}_T}(\Omega;H)$, we have for each $t\in [0,T)$, $\theta\in(0,1)$ such that there exists a constant $C>0$, the solution of \eqref{eq:main} satisfies
  \begin{equation}
    \label{lemma1-eq1}
    \E \|z(t ; \eta)\|_H^2 \leq e^{C\left[1+\left(\frac{1}{\theta(T-t)}  \right)^\frac{\gamma}{1-\gamma}\right]} \left(\E\|B^*z(t;\eta)\|_U^2\right)^{1-\theta} \left(\E \|\eta\|_H^2\right)^\theta.
  \end{equation}
\end{lemma}

\begin{proof}
 We only need to prove the case $\eta\neq0$. Set $z = z(\cdot;\eta)$, then it follows from the
 spectral-like condition ${\bf(H)}$ that
\begin{align*}
  \E \|\cE_\lambda z(t)\|_H^2
& \leq e^{N(1+\lambda^{\gamma})}\E\|B^*\cE_\lambda z(t)\|_U^2\\
& \leq e^{N(1+\lambda^{\gamma})}\big(\E \|B^*z(t)\|_U^2 +
\E\|B^*\cE^\bot_\lambda z(t)\|_U^2\big).
\end{align*}
Therefore, by the decay estimate
\eqref{eq:decay} we obtain that
\begin{align*}
&  \E \|z(t)\|_H^2
 = \E \|\cE_\lambda z(t)+\cE^\bot_\lambda
z(t)\|_H^2\\
 \le& 2(\E \|\cE_\lambda z(t)\|_H^2+\E \|\cE^\bot_\lambda z(t)\|_H^2)\\
\le&  e^{N(1+\lambda^{\gamma})}\big(\E \|B^*z(t)\|_U^2 +
\E\|B^*\cE^\bot_\lambda z(t)\|_U^2\big)+ \E
\|\cE^\bot_\lambda z(t)\|_H^2\\
\le& 2e^{N(1+\lambda^{\gamma})}\big(\E \|B^*z(t)\|_U^2 + \E
\|\cE^\bot_\lambda z(t)\|_H^2\big)\\
\le& 2e^{N(1+\lambda^{\gamma})}\big(\E \| B^*z(t)\|_U^2 +
e^{(-2\lambda+\tau)(T-t)} \E \|\eta\|_H^2\big)\\
\le&  2e^{N(1+\lambda^{\gamma})}e^{\tau T}\big(\E \|B^* z(t)\|_U^2 +
e^{-2\lambda(T-t)} \E \|\eta\|_H^2\big).
\end{align*}
Note that for each $t\in[0,T)$ and by the Young inequality, we have $\forall\,\epsilon>0$,
$$
N\lambda^{\gamma}=\frac{N}{[\epsilon(T-t)]^\gamma}[\epsilon\lambda(T-t)]^\gamma\leq\epsilon\lambda(T-t)+\left(\frac{N}{[\epsilon(T-t)]^\gamma}\right)
^{\frac{1}{1-\gamma}}.
$$
From these, we have for each $t\in[0,T)$ and $\forall\,\epsilon>0$,
\begin{equation}\label{lemma1-eq2}
\begin{split}
\E \|z(t)\|_H^2\leq& 2e^Ne^{\epsilon\lambda(T-t)}e^{\tau T}e^{\left(\frac{N}{[\epsilon(T-t)]^\gamma}\right)
^{\frac{1}{1-\gamma}}}\big(\E \|B^* z(t)\|_U^2 +
e^{-\lambda(T-t)} \E \|\eta\|_H^2\big)\\
=&2e^Ne^{\tau T}e^{\left(\frac{N}{[\epsilon(T-t)]^\gamma}\right)
^{\frac{1}{1-\gamma}}}\big(e^{\epsilon\lambda(T-t)}\E \|B^* z(t)\|_U^2 +
e^{-(1-\epsilon)\lambda(T-t)} \E \|\eta\|_H^2\big).
\end{split}
\end{equation}

Next, we claim that
\begin{equation}\label{claim}
\E\|z(t)\|_H^2\neq0, ~\text{for a.e.}~t\in[0,T]\Rightarrow \E\|B^*z(t)\|_U^2\neq0, ~\text{for a.e.}~t\in[0,T].
\end{equation}
Indeed, by contradiction, we suppose that $\chi_EB^*z(\cdot;\eta)= 0$, for some measurable sets $E\subset (0,T)$, such that $|E|>0$,
which, along with \eqref{lemma:main-ob} in Lemma \ref{lemma:main}, we have
$$
\E\|z(0)\|_H^2=0.
$$
This shows a contradiction. By the energy estimate for \eqref{eq:main} and assumption ${\bf (B)}$, there exists a positive constant $C_0$ such that the solution of \eqref{eq:main} satisfies
$$
\E\|B^*z(t)\|_U^2\leq \E\|z(t)\|_H^2\leq C_0\E\|\eta\|^2~\text{for a.e.}~t\in[0,T],
$$
which, along with $\eta\neq0$ and the above claim, we can choose
$$
\lambda=\frac{1}{T-t}\ln\frac{C_0\E\|\eta\|_H^2}{\E\|B^*z(t)\|_U^2}>0.
$$
This together with \eqref{lemma1-eq2} implies
\begin{equation*}
\begin{split}
\E \|z(t)\|_H^2\leq& 2e^Ne^{\tau T}e^{\left(\frac{N}{[\epsilon(T-t)]^\gamma}\right)
^{\frac{1}{1-\gamma}}}\left(C_0\E\|\eta\|_H^2\right)^\epsilon\left(\E\|B^*z(t)\|_U^2\right)^{1-\epsilon}\\
&+\left(\E\|B^*z(t)\|_U^2\right)^{1-\epsilon}\\
\leq&4C_0e^Ne^{\tau T}e^{\left(\frac{N}{[\epsilon(T-t)]^\gamma}\right)
^{\frac{1}{1-\gamma}}}\left(\E\|\eta\|_H^2\right)^\epsilon\left(\E\|B^*z(t)\|_U^2\right)^{1-\epsilon}\\
\leq&e^{C_2\left[1+\left(\frac{1}{\epsilon(T-t)}\right)^{\frac{\gamma}{1-\gamma}}\right]}\left(\E\|\eta\|_H^2\right)^\epsilon
\left(\E\|B^*z(t)\|_U^2\right)^{1-\epsilon},
\end{split}
\end{equation*}
where $C_2>0$. This is the desired estimate \eqref{lemma1-eq1} and the proof is completed.
\end{proof}

\begin{proposition}\label{po-1}
Suppose the assumptions ${\bf(H)}$ and ${\bf(B)}$ hold. Given any $\eta\in L^2_{\mathcal{F}_T}(\Omega;H)$, then
    we have for each $t\in [0,T)$, $\epsilon>0$ such that the solution of \eqref{eq:main} satisfies
  \begin{equation}
    \label{po1-eq1}
    \E \|z(t ; \eta)\|_H^2 \leq C_\gamma(t,\epsilon) \E\|B^*z(t;\eta)\|_U^2 +\epsilon\E \|\eta\|_H^2,
  \end{equation}
  where
  $$
  C_\gamma(t,\epsilon)=e^{C_3\left[1+\left(\frac{1}{T-t}\right)^{\frac{\gamma}{1-\gamma}}\right]}e^{\left[\frac{C_3}{T-t}\ln\left(e+\frac{1}{\epsilon}
  \right)\right]^\gamma},~\text{for some}~C_3>0.
  $$
\end{proposition}
\begin{proof}
Set $z = z(\cdot;\eta)$. From Lemma \ref{lemma-1}, we can rewrite \eqref{lemma1-eq1} in the following way:
$$
\E \|z(t)\|_H^2\leq \left(e^{\frac{C}{1-\theta}\left[1+\left(\frac{1}{\theta(T-t)}  \right)^\frac{\gamma}{1-\gamma}\right]} \E\|B^*z(t)\|_U^2\right)^{1-\theta} \left(\E \|\eta\|_H^2\right)^\theta.
$$
Now, by the following inequality; see \cite[Step 2, page 1144]{buffe2018spectral}: for any $a,b,c>0$ and $\theta\in(0,1)$,
$$
a\leq b^\theta c^{1-\theta}\Leftrightarrow a\leq\epsilon b+(1-\theta)\theta^{\frac{\theta}{1-\theta}}\epsilon^{-\frac{\theta}{1-\theta}}c,\,\forall\epsilon>0,
$$
choosing $a=\E \|z(t)\|_H^2$, $b=\E \|\eta\|_H^2$, $c=e^{\frac{C}{1-\theta}\left[1+\left(\frac{1}{\theta(T-t)}  \right)^\frac{\gamma}{1-\gamma}\right]} \E\|B^*z(t)\|_U^2$, we have
\begin{equation*}
\begin{split}
\E \|z(t)\|_H^2\leq& \varepsilon\E \|\eta\|_H^2+(1-\theta)\theta^{\frac{\theta}{1-\theta}}\epsilon^{-\frac{\theta}{1-\theta}}e^{\frac{C}{1-\theta}\left[1+\left(\frac{1}{\theta(T-t)}  \right)^\frac{\gamma}{1-\gamma}\right]} \E\|B^*z(t)\|_U^2\\
\leq& \varepsilon\E \|\eta\|_H^2+\epsilon^{-\frac{\theta}{1-\theta}}e^{\frac{C}{1-\theta}\left[1+\left(\frac{1}{\theta(T-t)}  \right)^\frac{\gamma}{1-\gamma}\right]} \E\|B^*z(t)\|_U^2.
\end{split}
\end{equation*}
By denoting $\beta=\frac{\theta}{1-\theta}$, it holds
$$
\E \|z(t)\|_H^2\leq \varepsilon\E \|\eta\|_H^2+\epsilon^{-\beta}e^{C(1+\beta)\left[1+\left(\frac{1+\beta}{\beta(T-t)}\right)^{\frac{\gamma}{1-\gamma}}\right]}\E\|B^*z(t)\|_U^2.
$$
Now, denote $a_0=K\left[1+\frac{1}{T-t}\right]^{\frac{\gamma}{1-\gamma}}$, $b_0=K\left(\frac{1}{T-t}\right)^{\frac{\gamma}{1-\gamma}}$, for some constant $K>0$, and note that
$$
\epsilon^{-\beta}e^{C(1+\beta)\left[1+\left(\frac{1+\beta}{\beta(T-t)}\right)^{\frac{\gamma}{1-\gamma}}\right]}\leq e^{a_0+\beta\left[\ln\left(e+\frac{1}{\epsilon}\right)+a_0  \right]+\left(\frac{1}{\beta}\right)^{\frac{\gamma}{1-\gamma}}b_0}.
$$
Choose $\beta=\left(\frac{b_0}{\ln\left(e+\frac{1}{\epsilon}\right)+a_0}\right)^{1-\gamma}$. Then for some constants $c_0, c_1>0$, we have
\begin{equation*}
\begin{split}
e^{a_0+\beta\left[\ln\left(e+\frac{1}{\epsilon}\right)+a_0  \right]+\left(\frac{1}{\beta}\right)^{\frac{\gamma}{1-\gamma}}b_0}\leq&
e^{c_0a_0+c_0\left[\ln\left(e+\frac{1}{\epsilon}\right)+a_0  \right]^\gamma b_0^{1-\gamma}}\\
\leq&e^{c_1a_0+c_1\left[\ln\left(e+\frac{1}{\epsilon}\right)  \right]^\gamma b_0^{1-\gamma}}.
\end{split}
\end{equation*}
Therefore, we obtain the desired inequality:
$$
\E \|z(t)\|_H^2 \leq e^{C_3\left[1+\left(\frac{1}{T-t}\right)^{\frac{\gamma}{1-\gamma}}\right]}e^{\left[\frac{C_3}{T-t}\ln\left(e+\frac{1}{\epsilon}
  \right)\right]^\gamma} \E\|B^*z(t;\eta)\|_U^2 +\epsilon\E \|\eta\|_H^2,
$$
where $C_3>0$. The proof is completed.
\end{proof}

\section{Impulse approximate controllability}
In this section, we shall show that the system \eqref{eq:main-for} is impulse approximate controllability.

Let $\tilde{T}\in(0,T)$, for any $l,\epsilon>0$, we define a functional on $L^2_{\mathcal{F}_T}(\Omega;H)$:
\begin{equation}
  \label{eq:J}
  \cJ(\eta) = \frac{l}{2} \E \|B^*z(T-\tilde{T};\eta)\|_U^2 + \frac{\epsilon}{2} \E \|\eta\|_H^2 - \E\langle y_0, z(0;\eta)\rangle_H,
\end{equation}
where $y_0\in L^2_{\mathcal{F}_0}(\Omega;H)$ is the initial state of \eqref{eq:main-for} and $z(\cdot;\eta)$ is the solution of \eqref{eq:main}.

First, we prove the existence of the minimizer of $\cJ$ defined in \eqref{eq:J}.
\begin{lemma}
  \label{lem:J_min}
  Suppose $y_0\in L^2_{\cF_0}(\Omega;H)$. Then there exists $\eta^\ast\in
  L^2_{\cF_T}(\Omega;H)$ such that
  $$
    \cJ(\eta^\ast) = \min_{\eta\in L^2_{\cF_T}(\Omega;H)} \cJ(\eta).
  $$
\end{lemma}
\begin{proof}
 One can easily check that the $\cJ(\cdot)$ is convex and continuous.

 Now, we prove the coercivity, let $\{\eta_n\}\subseteq L^2_{\cF_T}(\Omega;H)$ such that $$\E\|\eta_n\|^2_H\to \infty ~\text{as}~ n\to\infty,$$ and let
$$
  \wt{\eta}_n = \eta_n/ \E \|\eta_n\|_H^2,
$$
so that $\E \|\wt{\eta}_n\|_H^2 = 1$. Then
\begin{align*}
  \frac{\cJ(\eta_n)
  }{\E \|\eta_n\|_H^2} = \frac{l}{2}\E \|\eta_n\|_H^2  \E \|B^*z(T-\tilde{T};\wt{\eta}_n)\|_U^2 + \frac{\epsilon}{2} - \E\langle y_0, z(0;\wt{\eta}_n)\rangle_H.
\end{align*}
Note that
$$
  |\E\langle y_0, z(0;\wt{\eta}_n)\rangle_H| \leq (\E \|y_0\|_H^2)^{1/2} (\E
  \|z(0,\wt{\eta}_n)\|_H^2)^{1/2} \leq C(\E \|y_0\|_H^2)^{1/2}.
$$

If $\liminf_{n\to\infty} \E \|B^*z(T-\tilde{T};\wt{\eta}_n)\|_U^2>0$, we have
\begin{equation}
  \label{eq:coercivity}
  \cJ(\eta_n)
   \to \infty\quad\text{as}\quad \E\|\eta_n\|_H^2\to\infty,
\end{equation}
which implies the coercivity of $\cJ$.

If $\liminf_{n\to\infty} \E \|B^*z(T-\tilde{T};\wt{\eta}_n)\|_U^2 =0$, then it
follows from \eqref{lemma:main-ob} and the uniqueness of solution for \eqref{eq:main}, that $\wt{\eta}_n$ is bounded in $L^2_{\cF_T}(\Omega;H)$. Thus, we can extract a subsequence $\{\wt{\eta}_{n_j}\}\subseteq
\{\wt{\eta}_n\}$ such that $\wt{\eta}_{n_j}\rightharpoonup \wt{\eta}$ weakly in
$H$ and $z(\cdot; \wt{\eta}_{n_j})\rightharpoonup z(\cdot;\wt{\eta})$ weakly in
$C([0,T];H), ~\bP$-a.s. Moreover, by lower
semi-continuity we obtain
\[
   \E \|B^*z(T-\tilde{T};\wt{\eta})\|_U^2 \le \liminf_{j\to\infty} \E
  \|B^*z(T-\tilde{T}; \wt{\eta}_{n_j})\|_U^2 = 0.
\]
Then by \eqref{lemma:main-ob}, we get $z(0;\wt{\eta})=0,~\bP$-a.s., and thus
\begin{equation*}
\begin{split}
  \liminf_{n\to\infty}\frac{\cJ(\eta_n)
  }{\E \|\eta_n\|_H^2}
   \geq \liminf_{j\to\infty}\frac{\cJ(\eta_{n_j})
    }{\E \|\eta_{n_j}\|_H^2}
   \geq& \liminf_{j\to \infty} [\frac{\epsilon}{2} - \E\langle y_0, z(0;\wt{\eta}_{n_j})\rangle_H]\\
   =&\frac{\epsilon}{2} - \E\langle y_0, z(0,\wt{\eta})\rangle_H =\frac{\epsilon}{2},
   \end{split}
\end{equation*}
which implies~\eqref{eq:coercivity}, and so $\cJ$ is coercive.

To sum up, we showed that $\cJ$ is convex, continuous, and coercive, and thus
the minimizer of $\cJ$ exists.
\end{proof}
Next, we present several equivalent inequalities.

\begin{theorem}\label{th-1}
Let $0<\tilde{T}<T\leq2\tilde{T}$ and $l,\epsilon>0$. For any $\eta\in L^2_{\cF_T}(\Omega;H)$, the following two statements are equivalent.

$(i)$: $\forall\,y_0\in L^2_{\cF_0}(\Omega;H)$, there exists a control $u\in L^2_{\cF_{\tilde{T}}}(\Omega;U)$ such that the solution of \eqref{eq:main-for} satisfies
$$
\frac{1}{l}\E\|u\|_U^2+\frac{1}{\epsilon}\E\|y(T)\|_H^2\leq\E\|y_0\|_H^2.
$$

$(ii)$: The solution of \eqref{eq:main} satisfies
$$
\E\|z(0)\|_H^2\leq l\E\|B^*z(T-\tilde{T};\eta)\|_U^2+\epsilon\E\|\eta\|_H^2.
$$
\end{theorem}
\begin{proof}
Set $z = z(\cdot;\eta)$. We first prove $(i)\Rightarrow(ii)$.

For system \eqref{eq:main-for}, system \eqref{eq:main} replaced $t$ by $T-t$ and \Ito formulate, we have
\begin{equation*}
\begin{split}
d\langle y(T-t),z(T-t)\rangle_H=&\langle dy(T-t),z(T-t)\rangle_H+\langle y(T-t),dz(T-t)\rangle_H\\
&+\langle dy(T-t),dz(T-t)\rangle_H.
\end{split}
\end{equation*}
Then integrating from $(0,T)$ and taking the expectation, we obtain
\begin{equation}\label{th1-1}
\begin{split}
\E\langle y_0,z(0)\rangle_H-\E\langle y(T),\eta\rangle_H=\E\langle u,B^*z(T-\tilde{T})\rangle_U.
\end{split}
\end{equation}
By Cauchy-Schwarz's inequality and using the inequality in $(i)$, one can deduce that
\begin{equation*}
\begin{split}
&\E\langle y_0,z(0)\rangle_H\\
\leq&\left(\E\|u\|_U^2\right)^{\frac{1}{2}}\left(\E\|B^*z(T-\tilde{T})\|_U^2\right)^{\frac{1}{2}}
+\left(\E\|y(T)\|_H^2\right)^{\frac{1}{2}}\left(\E\|\eta\|_H^2\right)^{\frac{1}{2}}\\
\leq&\frac{1}{2l}\E\|u\|_U^2+\frac{1}{2\epsilon}\E\|y(T)\|_H^2+\frac{l}{2}\E\|B^*z(T-\tilde{T})\|_U^2+\frac{\epsilon}{2}\E\|\eta\|_H^2\\
\leq&\frac{1}{2}\E\|y_0\|_H^2+\frac{1}{2}\left[l\E\|B^*z(T-\tilde{T})\|_U^2+\epsilon\E\|\eta\|_H^2\right],
\end{split}
\end{equation*}
which, along with choosing $y_0=z(0)$, it implies
$$
\E\|z(0)\|_H^2\leq l\E\|B^*z(T-\tilde{T};\eta)\|_U^2+\epsilon\E\|\eta\|_H^2.
$$

Next, we prove $(ii)\Rightarrow(i)$.

From Lemma \ref{lem:J_min}, we know that the functional $\cJ$ has a minimizer $\eta^\ast\in L^2_{\cF_T}(\Omega;H)$, i.e.,
$$
    \cJ(\eta^\ast) = \min_{\eta\in L^2_{\cF_T}(\Omega;H)} \cJ(\eta).
  $$
Consider the following backward systems:
\begin{equation*}
\left\{
\begin{array}{ll}
d \tilde{z} = -A \tilde{z}dt - F(t)\tilde{Z} dt + \tilde{Z} dW(t), & t\in(0,T),   \\[2mm]
\tilde{z}(T) =\eta^*,
\end{array}
\right.
\end{equation*}
and
\begin{equation*}
\left\{
\begin{array}{ll}
d h = -A hdt - F(t)H dt + H dW(t), & t\in(0,T),   \\[2mm]
h(T) =h_T,
\end{array}
\right.
\end{equation*}
where $\eta^*$ is the minimizer of the functional $\cJ$ and $h_T\in L^2_{\mathcal{F}_T}(\Omega;H)$. Noting that
$$
\cJ'(\eta^*)h_T=0,~\text{for any}~h_T\in L^2_{\mathcal{F}_T}(\Omega;H),
$$
we have for any $h_T\in L^2_{\mathcal{F}_T}(\Omega;H)$,
\begin{equation}
\label{th1-2}
l\E\langle B^*\tilde{z}(T-\tilde{T}),B^*h(T-\tilde{T})\rangle_U+\epsilon\E\langle\eta^*,h_T\rangle_H-\E\langle y_0,h(0)\rangle_H=0.
\end{equation}
On one hand, similar to \eqref{th1-1}, we obtain
\begin{equation}\label{th1-3}
\begin{split}
\E\langle u,B^*h(T-\tilde{T})\rangle_U-\E\langle y_0,h(0)\rangle_H+\E\langle y(T),h_T\rangle_H=0.
\end{split}
\end{equation}
Noting that $0<\tilde{T}<T\leq2\tilde{T}$, by choosing $$u=lB^*\tilde{z}(T-\tilde{T})\in L^2_{\mathcal{F}_{T-\tilde{T}}}(\Omega;U)\subset L^2_{\mathcal{F}_{\tilde{T}}}(\Omega;U)$$ and along with \eqref{th1-2} and \eqref{th1-3}, it implies that the solution of \eqref{eq:main-for} satisfies:
$$
y(T)=\epsilon\eta^*.
$$
Furthermore, after calculation, we have
\begin{equation}\label{th1-4}
\begin{split}
l\E\|B^*\tilde{z}(T-\tilde{T})\|_U^2+\epsilon\E\|\eta^*\|_H^2=\frac{1}{l}\E\|u\|_U^2+\frac{1}{\epsilon}\E\|y(T)\|_H^2.
\end{split}
\end{equation}

On the other hand, taking $h(t)=\tilde{z}(t)$ for any $t\in[0,T]$, by \eqref{th1-2}, we get
$$
l\E\|B^*\tilde{z}(T-\tilde{T})\|_U^2+\epsilon\E\|\eta^*\|_H^2-\E\langle y_0,\tilde{z}(0)\rangle_H=0,
$$
which, along with Cauchy-Schwarz's inequality and the inequality of $(ii)$, it holds
\begin{equation*}
\begin{split}
&l\E\|B^*\tilde{z}(T-\tilde{T})\|_U^2+\epsilon\E\|\eta^*\|_H^2\leq\left(\E\|y_0\|_H^2\right)^{\frac{1}{2}}\left(\E\|\tilde{z}(0)\|_H^2\right)^{\frac{1}{2}}\\
\leq&\left(\E\|y_0\|_H^2\right)^{\frac{1}{2}}\left[l\E\|B^*\tilde{z}(T-\tilde{T})\|_U^2+\epsilon\E\|\eta^*\|_H^2\right]^{\frac{1}{2}}.
\end{split}
\end{equation*}
This implies
$$
l\E\|B^*\tilde{z}(T-\tilde{T})\|_U^2+\epsilon\E\|\eta^*\|_H^2\leq\E\|y_0\|_H^2,
$$
which, together with \eqref{th1-4}, it holds
$$
\frac{1}{l}\E\|u\|_U^2+\frac{1}{\epsilon}\E\|y(T)\|_H^2\leq\E\|y_0\|_H^2.
$$
The proof is completed.
\end{proof}

Now, we are ready to prove Theorem \ref{thm:main}.
\begin{proof}[proof of Theorem \ref{thm:main}]
Set $z=z(\cdot;\eta)$. From Proposition \ref{po-1} and the energy estimate of solution for \eqref{eq:main}, we have that for some constant $C>0$, the solution of \eqref{eq:main} satisfies
\begin{equation*}
\begin{split}
    \E \|z(0)\|_H^2\leq& C\E \|z(T-\tilde{T})\|_H^2\\
    \leq&C_\gamma(T-\tilde{T},\epsilon)C \E\|B^*z(T-\tilde{T})\|_U^2 +\epsilon\E \|\eta\|_H^2.
\end{split}
\end{equation*}
Now, by Theorem \ref{th-1} with $l=C_\gamma(T-\tilde{T},\epsilon)C$ and along with the above inequality, we have that the solution of \eqref{eq:main-for} and the control $u$ satisfies
$$
\frac{1}{l}\E\|u\|_U^2+\frac{1}{\epsilon}\E\|y(T)\|_H^2\leq\E\|y_0\|_H^2.
$$
This implies the desired estimates \eqref{thm:main-ob} and \eqref{thm:main-u}. The proof is completed.
\end{proof}
\section{Norm optimal control problem}
In this section, based on Theorem \ref{thm:main}, we shall prove Theorem \ref{thm:main2}.
\begin{proof}[proof of Theorem \ref{thm:main2}] At first, by Theorem \ref{thm:main}, we see that problem ${\bf(NP)}_{y_0}^T$ has an admissible control. So there exist a sequence $\{ u_n^* \}_{n\geq1}\subseteq L^2_{\cF_{T}}(\Omega;U)$ holding
\begin{equation}\label{proposition:TN-exist-eq2}
\E\|u_n^*\|_U^2\leq N(T,y_0)+\frac{1}{n}~\text{and}~\E\|y(T;y_0,u_n^*)\|_H^2\leq\epsilon\E\|y_0\|_H^2,
\end{equation}
where $y(\cdot;y_0,u_n^*)$ is the solution of \eqref{eq:main-for} with the control $u_n^*$.
By the first inequality in (\ref{proposition:TN-exist-eq2}), there exists a subsequence of $\{ u_n^* \}_{n\geq1}$, denoted in the same manner, and
$u^*\in L^2_{\cF_{T}}(\Omega;U)$ such that
\begin{equation}
\label{proposition:TN-exist-lim2}
u_n^*\rightarrow u^*\quad\ \text{weakly in}~\,L^2_{\cF_{T}}(\Omega;U)\quad\ \text{as}~\,n\rightarrow\infty.
\end{equation}
and
\begin{equation}
\label{proposition:TN-exist-lim28}
\E\|u^*\|_U^2\leq\liminf\limits_{n\rightarrow\infty}\E\|u_n^*\|_U^2\leq N(T,y_0).
\end{equation}

Consider the following equation:
\begin{equation}
\label{model8}
\left\{
\begin{array}{ll}
dy_n = A y_ndt + F(t)y_n dW(t), & D\times(0,T)\setminus\{\tilde{T}\},   \\[2mm]
y_n(0) =y_0,\,y_n(\tilde{T})=y_n(\tilde{T}^-)+Bu_n^*,
\end{array}
\right.
\end{equation}
which admits a solution $y_{n} = y(\cdot;y_0,u_n^*)$.
Let $y=y(\cdot; y_0, u^*)$ satisfies the following system:
\begin{equation}
    \label{model9}
\left\{
\begin{array}{ll}
dy = A ydt + F(t)y dW(t), & D\times(0,T)\setminus\{\tilde{T}\},   \\[2mm]
y(0) =y_0,\,y(\tilde{T})=y(\tilde{T}^-)+Bu^*,
\end{array}
\right.
\end{equation}
According to the assumption ${\bf(S)}$ with \eqref{proposition:TN-exist-lim2}, we have
\begin{equation}\label{79}
y(\cdot; y_0, u_n^*)\rightarrow y(\cdot; y_0, u^*)~\text{strongly in}~C([0,T];H), \,\bP-a.s.
\end{equation}
This, together with the second inequality of \eqref{proposition:TN-exist-eq2}, implies that
\begin{equation}\label{74}
\E\|y(T;y_0,u^*)\|_H^2\leq\epsilon\E\|y_0\|_H^2,
\end{equation}
which shows that $u^*$ is also an  admissible control. Then we have
$$
N(T,y_0)\leq \E\|u^*\|_U^2.
$$
This together with \eqref{proposition:TN-exist-lim28} and \eqref{74} implies $u^*$ is a optimal control to problem ${\bf (NP)_{y_0}^T}$.

Next, we prove the uniqueness. By contradiction, assume there are two different norm optimal controls $u_1^*,u_2^*\in L^2_{\cF_{T}}(\Omega;U)$ such that the solution of \eqref{eq:main-for} satisfies
$$
\E\|y(T;y_0,u_1^*)\|_H^2\leq\epsilon\E\|y_0\|_H^2~\text{and}~\E\|y(T;y_0,u_2^*)\|_H^2\leq\epsilon\E\|y_0\|_H^2.
$$
Then $w = (u_1^* + u_2^*)/2\in L^2_{\cF_{T}}(\Omega;U)$ is also a norm optimal control  such that the solution of \eqref{eq:main-for} satisfies
$$
\E\|y(T;y_0,w)\|_H^2\leq\epsilon\E\|y_0\|_H^2.
$$
Since
$$
u_1^*\neq u_2^*~\text{for}~a.e.x\in U,
$$
which along with the Hilbert space $L^2_{\cF_{T}}(\Omega;U)$ is a strictly convex space (see \cite[page 34]{lv1}) (i.e., let $Y$ is a strictly convex space and $u,v\in Y$, if $\|u\|_Y=\|v\|_Y=1$ and $u\neq v$ for $a.e.x\in Y$, it shows that $\|u+v\|_Y<2$), it implies that $$\E\|w\|_U^2<N(T,y_0),$$ which contradicts with \eqref{N}.  The proof is completed.
\end{proof}

\section{Time optimal control problem}
In this section, we shall show that the time optimal control problem ${\bf (TP)}$ has a solution.

We start with the existence of optimal controls to the problem ${\bf(TP)}$ for any $\tilde{T}>0$, $M>0$ and any
$\epsilon> 0$.

\begin{lemma}\label{lemma-9}
Suppose the assumptions ${\bf(H)}$, ${\bf(B)}$ and ${\bf(S)}$ hold. Let $\tilde{T}>0$, $M>0$ and $\epsilon> 0$. Then the problem ${\bf(TP)}$ has at lease one optimal control.
\end{lemma}
\begin{proof}~From Theorem \ref{thm:main}, we see that there exists an admissible control to the problem ${\bf(TP)}$.

We next claim that
\begin{equation}\label{Proof-3}
{\bf(TP)}\;\;\mbox{has at least one optimal control}.
\end{equation}

For this purpose, on one hand, according to (\ref{impulse-2(1)}), there exist  sequences
$\{T_n\}_{n\geq 1}\subseteq (\tilde{T},+\infty)$  and
$\{u_n\}_{n\geq 1}\subseteq \mathcal{U}_M$ so that
\begin{equation}\label{Proof-4}
T_n\rightarrow T^*\in (\tilde{T},+\infty)\;\;\mbox{and}\;\;\{y(T_n; y_0, u_n)\}_{n\geq 1}\subseteq B_\epsilon(0).
\end{equation}
On the other hand, since $\{u_n\}_{n\geq 1}\subseteq \mathcal{U}_M$, there exists a control
$\widetilde{u}\in \mathcal{U}_M$ and a subsequence of $\{u_n\}_{n\geq 1}$, still denoted in the same manner,
so that
\begin{equation}\label{Proof-5}
u_n\rightarrow \widetilde{u}\;\;\mbox{weakly in}\;\;L^2_{\cF_T}(\Omega;L^2(D)),
\end{equation}
which, along with the first relation of (\ref{Proof-4}) and according to the assumption ${\bf(S)}$ with \eqref{Proof-5}, similar to \eqref{79}, we can take the limit for $n\rightarrow +\infty$ to obtain that
\begin{equation*}
y(T_n; y_0, u_n)\rightarrow y(T^*; y_0, \widetilde{u})~\text{strongly in}~C([0,T];H), \,\bP-a.s.
\end{equation*}
This, along with the second relation of (\ref{Proof-4}), implies that
\begin{equation*}
y(T^*; y_0, \widetilde{u})\in B_\epsilon(0).
\end{equation*}
Hence, $\widetilde{u}$ is an optimal control
to ${\bf(TP)}$, i.e., (\ref{Proof-3}) follows. In summary, we finish the proof of this lemma.
\end{proof}
We now turn to the proof of Theorem~\ref{impulse-4}.

\begin{proof}[Proof of Theorem~\ref{impulse-4}] We divide the proof into three Steps.

Step 1. Let $\tilde{T}>0$, $M>0$ and $\epsilon> 0$.
According to Lemma~\ref{lemma-9}, we see that
$T^*>\tilde{T}$ and ${\bf (TP)}$ has an optimal control
$u^*$.

Step 2. We prove that any  optimal control $u^*$ satisfies \eqref{tu}.

 Arbitrarily fix $u_1,u_2\in\cU_{M}$, and $\alpha\in[0,1]$. Let
  $$
  \left\{
  \begin{array}{ll}
  u_\alpha\triangleq(1-\alpha)u_1+\alpha u_2;\\[3mm]
  y_1(t)\triangleq y(t; y_0,u_1),\,\,t\in[0,T];\\[3mm]
  y_2(t)\triangleq y(t; y_0,u_2),\,\,t\in[0,T];\\[3mm]
  y_\alpha(t)\triangleq y(t; y_0,u_\alpha),\,\,t\in[0,T].
  \end{array}
  \right.
  $$
  One can easily check that
  $$
  \begin{array}{lll}
  y_\alpha(T)=y_1(T)+y(T;0,\alpha(u_2-u_1)).
  \end{array}
  $$
  Then, after calculation, we have
  \begin{equation}
  \label{lemma-OP-eq6}
  \begin{array}{lll}
  &\E\|y_\alpha(T)\|_H^2-\E\|y_1(T)\|_H^2\\[2mm]
  =&\E\|y(T;0,\alpha(u_2-u_1))\|_H^2 +2\E\big\langle y(T;0,\alpha(u_2-u_1)),y_1(T)\big\rangle_H.
  \end{array}
  \end{equation}

Let $u^*$ be an optimal control to ${\bf (TP)}$. Arbitrarily take $u\in \cU_{M}$. Then choosing $u_1=u^*$, $u_2=u$ in (\ref{lemma-OP-eq6}), we obtain that for all $u\in \cU_{M}$,
  $$
  \begin{array}{lll}
  &&\alpha^2\E\|y(T;0,u-u^*)\|_H^2   +2\alpha\E\big\langle y(T;0,u-u^*),y(T; y_0,u^*)\big\rangle_H\geq0,
  \end{array}
  $$
  which, implies that for all $u\in\cU_{M}$
\begin{equation}
\label{lemma-OP-eq99}
\E\langle y(T;0,u-u^*),y(T; y_0,u^*)\rangle_H\geq0.
\end{equation}
For system \eqref{eq:main-for2} (with $y_0=0$, $u^*$ replaced by $u-u^*$), system \eqref{eq:main} (with $\eta=y(T; y_0,u^*)\neq0$) and applying It\^{o}'s formula, we obtain
\begin{equation}\label{32}
\E \left\langle u-u^*, B^* z(\tilde{T})\right\rangle_U
=-\E\langle y(T;0,u-u^*),y(T; y_0,u^*)\rangle_H.
\end{equation}

This, along with (\ref{lemma-OP-eq99}) stands that
$$
\E \left\langle u-u^*, B^* z(\tilde{T})\right\rangle_U\leq0, \text{for each}~ u\in\cU_{M},
$$
  which implies
\begin{equation}\label{33}
\E\langle B^* z(\tilde{T}),u^* \rangle_U =\max_{u\in \cU_M} \E\langle B^* z(\tilde{T}),u \rangle_U .
\end{equation}
This shows that
$$
\E\langle B^* z(\tilde{T}),u^* \rangle_U \leq M\left(\E\|B^* z(\tilde{T})\|_U^2\right)^{\frac{1}{2}}.
$$
Noting that $\eta\neq0$, by the claim \eqref{claim}, we can denote
\begin{equation}\label{u0}
u_0=M\frac{B^* z(\tilde{T})}{\left(\E\|B^* z(\tilde{T})\|_U^2\right)^{\frac{1}{2}}}.
\end{equation}
It implies $u_0\in\cU_M$. Also, we have
$$
\begin{array}{lll}
\E\langle B^*z(\tilde{T}),u_0 \rangle_U&=&\E\left\langle B^* z(\tilde{T}),M\frac{B^* z(\tilde{T})}{\left(\E\|B^* z(\tilde{T})\|_U^2\right)^{\frac{1}{2}}}\right\rangle_U\\
&=&M\left(\E\|B^* z(\tilde{T})\|_U^2\right)^{\frac{1}{2}},
\end{array}
$$
which, along with (\ref{33}), it holds
$$
\E\langle B^* z(\tilde{T}),u^* \rangle_U=M\left(\E\|B^* z(\tilde{T})\|_U^2\right)^{\frac{1}{2}}.
$$
Therefore, we have
$$
u^*=M\frac{B^* z(\tilde{T})}{\left(\E\|B^*z(\tilde{T})\|_U^2\right)^{\frac{1}{2}}}.
$$
This implies \eqref{tu}.

Step 3. We prove the uniqueness.

Suppose that $v^*$ is also an optimal control to ${\bf (TP)}$.
 It is clear that
$(u^*+v^*)/2$ is an optimal control to ${\bf (TP)}$. According to \eqref{tu},
\begin{equation*}
\E\|u^*\|_U^2=\E\|v^*\|_U^2=\E\|(u^*+v^*)/2\|_U^2=M^2.
\end{equation*}
These, together with the parallelogram rule, yield that
\begin{equation*}
\E\|u^*-v^*\|_U^2=2\big(\E\|u^*\|_U^2
+\E\|v^*\|_U^2\big)-\E\|u^*+v^*\|_U^2=0.
\end{equation*}
Hence, $u^*=v^*$. In summary, we finish the proof of Theorem~\ref{impulse-4}.
\end{proof}

\section{Examples}
In this Section, we present some specific equations which are covered by \eqref{eq:main-for} and satisfy the assumption ${\bf (H)}$, ${\bf (B)}$ and ${\bf (S)}$. 

Let $D$ be a bounded domain of $\R^d$, $d\geq1$ with boundary $\partial D$ of class $C^2$ and $G\subset D$ be a measurable and nonempty subset with positive measure and denote by $\chi_G$ the characteristic function of $G$. Throughout this Section, we denote by $\langle\cdot,\cdot\rangle$ the scalar product in $L^2(D)$ and denote by $\|\cdot\|$ the norm induced by $\langle\cdot,\cdot\rangle$.

{\bf Example 7.1.} Let $H=U:=L^2(D)$ and the Laplacian with $A=-\Delta$ with $D(A) =H^2(D)\cap H^1_0(D)$. We defined $B:U\rightarrow H$ in the following manner:
$$
Bu:=\chi_Gu,~\text{for all}~u\in L^2_{\cF_T}(\Omega;U).
$$
Then the assumption ${\bf (H)}$ holds with $\gamma=\frac{1}{2}$; see \cite{apraiz2014observability}. Obviously, the assumption ${\bf (B)}$ stands. The corresponding evolution equation \eqref{eq:main-for} is as follows:
$$
\left\{
\begin{array}{ll}
dy = -\Delta ydt + F(t)y dW(t), & D\times(0,T)\setminus\{\tilde{T}\},   \\[2mm]
y(0) =y_0,\,y(\tilde{T})=y(\tilde{T}^-)+\chi_Gu.
\end{array}
\right.
$$  
It is well known that $y\in L^2_{\mathbb{F}}(\Omega;C([0,T];L^2(D))) \cap L^2_{\mathbb{F}}(0,T;H_0^1(D)))$.

Next, we show that in this example, the assumption ${\bf (S)}$ holds. 

Indeed, consider the following equation:
\begin{equation}
\label{model28}
\left\{
\begin{array}{ll}
dy_n = -\Delta y_ndt + F(t)y_n dW(t), & D\times(0,T)\setminus\{\tilde{T}\},   \\[2mm]
y_n(0) =y_0,\,y_n(\tilde{T})=y_n(\tilde{T}^-)+\chi_Gu_n,
\end{array}
\right.
\end{equation}
which admits a solution $y_{n} = y(\cdot;y_0,u_n)$, and there is a $u\in L^2_{\cF_{T}}(\Omega;L^2(D))$ such that
\begin{equation}
\label{proposition:TN-exist-lim22}
u_n\rightarrow u\quad\ \text{weakly in}~\,L^2_{\cF_{T}}(\Omega;L^2(D))\quad\ \text{as}~\,n\rightarrow\infty.
\end{equation} 
Let $y=y(\cdot; y_0, u)$ satisfies the following system:
\begin{equation}
    \label{model29}
\left\{
\begin{array}{ll}
dy = -\Delta ydt + F(t)y dW(t), & D\times(0,T)\setminus\{\tilde{T}\},   \\[2mm]
y(0) =y_0,\,y(\tilde{T})=y(\tilde{T}^-)+\chi_Gu,
\end{array}
\right.
\end{equation}
Denote $\bar{y}_n=y_n-y$. By equation \eqref{model28} and equation \eqref{model29}, we have
\begin{equation}
    \label{model102}
\left\{
\begin{array}{ll}
d\bar{y}_n = -\Delta \bar{y}_ndt + F(t)\bar{y}_n dW(t), & D\times(0,T)\setminus\{\tilde{T}\},   \\[2mm]
\bar{y}_n(0) =0,\,\bar{y}_n(\tilde{T})=\bar{y}_n(\tilde{T}^-)+\chi_G(u_n-u),
\end{array}
\right.
\end{equation}
Similar to \cite[(9.8) in page 276]{lv1}, for a positive constant $C$ independent of $n$, we have the energy estimate of the solution:
\begin{equation}
\label{ineq-y-2}
\|\bar{y}_n\|_{L^2_{\mathbb{F}}(\Omega;C([0,T];L^2(D))) \cap L^2_\mathbb{F}(0,T;H_0^1(D))}
\leq C\E\|\chi_G(u_n-u)\|^2,
\end{equation}
which, along with \eqref{proposition:TN-exist-lim22} implies $\bar{y}_n$ is bounded in $L^2_{\mathbb{F}}(\Omega;C([0,T];L^2(D))) \cap L^2_\mathbb{F}(0,T;H_0^1(D))$. Then there exists a $\bar{y}\in L^2_{\mathbb{F}}(\Omega;C([0,T];L^2(D))) \cap L^2_\mathbb{F}(0,T;H_0^1(D))$ such that
\begin{equation}\label{882}
 \bar{y}_n\rightarrow \bar{y}~\text{weakly in}~  L^2_{\mathbb{F}}(\Omega;C([0,T];L^2(D))) \cap L^2_\mathbb{F}(0,T;H_0^1(D)) ~\text{as}~n\rightarrow\infty.
\end{equation}
This together with the embedding $H_0^1(D)\hookrightarrow L^2(D)$ is compact; see \cite[Section 5.7 Theorem 1]{evans2022partial}, implies that there exists a subsequence of $\{\bar{y}_n\}_{n\geq1}$ denoted itself converging in $L^2(D)$ for $a.e.\,t\in(0,T)$, $\bP-a.s$., to some limit $\bar{y}(t)$:
$$
\lim_{n\rightarrow\infty}\|\bar{y}_n(t)-\bar{y}(t)\|=0, ~\text{for a.e.}~t\in(0,T),\,\bP-a.s.,
$$
which implies
\begin{equation}
    \label{892}
    \bar{y}_n\rightarrow \bar{y}~\text{strongly in}~C([0,T];L^2(D)), \,\bP-a.s.
\end{equation}
Passing to the limit for $n\rightarrow\infty$ in \eqref{model102}, by \eqref{proposition:TN-exist-lim22} and \eqref{882}, we obtain
\begin{equation}
    \label{model112}
\left\{
\begin{array}{ll}
d\bar{y} = -\Delta \bar{y}dt + F(t)\bar{y} dW(t), & D\times(0,T)\setminus\{\tilde{T}\},   \\[2mm]
\bar{y}(0) =0,\,\bar{y}(\tilde{T})=\bar{y}(\tilde{T}^-),
\end{array}
\right.
\end{equation}
Thus by the uniqueness of the solution of \eqref{model112}, we have $\bar{y}=0$ in $D$ $\bP$-a.s. This along with \eqref{892}, implies that $\bar{y}_n\rightarrow0$ strongly in $C([0,T];L^2(D))$, $\bP$-a.s.,  i.e.,
$$
y(\cdot; y_0, u_n)\rightarrow y(\cdot; y_0, u)~\text{strongly in}~C([0,T];L^2(D)), \,\bP-a.s.,
$$ 
i.e., the assumption ${\bf (S)}$ holds. 

{\bf Example 7.2.} Let $D:=(0,1)$ and $H=U:=L^2(D)$ and the one-dimensional degenerate operator with $A=-\partial_x(x^\alpha\partial_x)$ with $\alpha\in(0,2)$, and $\mathcal{D}(A):=\{v\in H_\alpha^1(D):(x^{\alpha}v_x)_x\in L^2(D) ~\text{and}~ BC_\alpha(v) =0 \}$, where
$$
\begin{array}{ll}
H^1_\alpha(D):=\bigg\{v\in L^2(D):& v \text{ is absolutely continuous in}~ D,
 x^{\frac{\alpha}{2}}v_x\in L^2(D)\\
 \quad & \text{and}~\,v(1)=0 \bigg\},
\end{array}
$$
and
\begin{equation*}
BC_\alpha(v)=
\begin{cases}
v_{|_{x=0}}, &\alpha\in(0,1),\\
(x^{\alpha}v_x)_{|_{x=0}}, &\alpha\in[1,2),
\end{cases}
\end{equation*}
endowed with the norms
$
\|v\|^2_{H^1_\alpha(D)}:=\|v\|^2+\|\sqrt{x^\alpha}v_x\|^2.
$
 
 We defined $B:U\rightarrow H$ in the following manner:
$$
Bu:=\chi_Gu,~\text{for all}~u\in L^2_{\cF_T}(\Omega;U).
$$
Then the assumption ${\bf (H)}$ holds with $\gamma\in(0,1)$ is defined as follows; see \cite{liu2023}.
$$
\gamma=
\left\{
\begin{array}{lll}
\dfrac{3}{4}, &~\text{if}~ \alpha\in(0,2)\setminus\{1\},\\[3mm]
\dfrac{3}{2\sigma}  ~\text{for any}~\sigma\in(0,2),&~\text{if}~ \alpha=1.
\end{array}
\right.
$$ 
Obviously, the assumption ${\bf (B)}$ stands. Similar to the discussion in example 7.1 and along with the embedding $H_\alpha^1(D)\hookrightarrow L^2(D)$ is compact; see \cite{alabau2006carleman}, the assumption ${\bf (S)}$ is valid for this case.

\section{Declarations}
The authors have not disclosed any competing interests.

\bibliographystyle{abbrvnat}
\bibliography{ref.bib}

\end{document}